\documentclass[reqno]{amsart}
\usepackage{amsrefs}
\usepackage[shortlabels]{enumitem}
\usepackage{tikz-cd}
\usepackage{hyperref}

\newcommand{\vertiii}[1]{{\left\vert\kern-0.25ex\left\vert\kern-0.25ex\left\vert #1 \right\vert\kern-0.25ex\right\vert\kern-0.25ex\right\vert}}
\DeclareMathOperator{\Id}{Id}
\DeclareMathOperator{\ran}{ran}
\DeclareMathOperator{\tr}{tr}
\newcommand{\Fourier}{\mathcal{F}}
\newcommand{\Inverse}{\hat{\mathcal{F}}}
\newcommand{\RKHS}{\mathcal{H}}
\newcommand{\cmplx}{\mathbb{C}}
\newcommand{\norm}[1]{\left\| #1 \right\|}
\newcommand{\abs}[1]{\left| #1 \right|}

\theoremstyle{plain}
\newtheorem{theorem}{Theorem}
\newtheorem{lemma}[theorem]{Lemma}
\newtheorem{proposition}[theorem]{Proposition}
\newtheorem{corollary}[theorem]{Corollary}

\theoremstyle{definition}

\theoremstyle{remark}
\newtheorem{remark}{Remark}

\title[]{On harmonic Hilbert spaces on compact abelian groups}
\author{Suddhasattwa Das}
\address{Department of Mathematics and Statistics, Texas Tech University, Lubock, TX 79409, USA.}
\email{suddas@ttu.edu}
\author{Dimitrios Giannakis}
\address{Department of Mathematics, Dartmouth College, Hanover, NH 03755, USA.}
\email{dimitrios.giannakis@dartmouth.edu}

\begin{document}

\begin{abstract} 
    Harmonic Hilbert spaces on locally compact abelian groups are reproducing kernel Hilbert spaces (RKHSs) of continuous functions constructed by Fourier transform of weighted $L^2$ spaces on the dual group. It is known that for suitably chosen subadditive weights, every such space is a Banach algebra with respect to pointwise multiplication of functions. In this paper, we study RKHSs associated with subconvolutive functions on the dual group. Sufficient conditions are established for these spaces to be symmetric Banach $^*$-algebras with respect to pointwise multiplication and complex conjugation of functions (here referred to as RKHAs). In addition, we study aspects of the spectra and state spaces of RKHAs. Sufficient conditions are established for an RKHA on a compact abelian group $G$ to have the same spectrum as the $C^*$-algebra of continuous functions on $G$. We also consider one-parameter families of RKHSs associated with semigroups of self-adjoint Markov operators on $L^2(G)$, and show that in this setting subconvolutivity is a necessary and sufficient condition for these spaces to have RKHA structure. Finally, we establish embedding relationships between RKHAs and a class of Fourier--Wermer algebras that includes spaces of dominating mixed smoothness used in high-dimensional function approximation.            
\end{abstract}

\maketitle

\section{Introduction} 
\label{sec:intro}

The notion of a harmonic Hilbert space on the real line was introduced by Delvos~\cite{Delvos97}, extending earlier work of Babu\v ska \cites{Babuska68a,Babuska68b} and Pr\'ager \cite{Prager79} on periodic Hilbert spaces. Subsequently, Feichtinger, Pandey, and Werther \cite{FeichtingerEtAl07} constructed harmonic Hilbert spaces on arbitrary locally compact abelian groups (LCAs), and characterized their properties using the theory of Wiener amalgam spaces \cite{Feichtinger80}. In essence, a harmonic Hilbert space $\mathcal H_\lambda$ on an LCA $G$ is the image under the Fourier transform of a weighted $L^2$ space on the dual group, $\hat G$, associated with a weight function $ w \equiv  \lambda^{-1/2} : \hat G \to \mathbb R_+$, where $\lambda \in L^1(\hat G)$. Harmonic Hilbert spaces are reproducing kernel Hilbert spaces (RKHSs) of continuous functions, and have a number of useful properties for function approximation on LCAs \cites{Delvos97,Delvos02,FeichtingerEtAl07}. In particular, using results in \cite{Feichtinger79}, it has been shown that if $w$ is subadditive (up to a constant $C$), 
\begin{equation}
    \label{eqSubadd}
    w(\gamma + \gamma') \leq C (w(\gamma) + w(\gamma')), \quad \gamma,\gamma' \in \hat G,
\end{equation}
then $\mathcal H_\lambda$ is a Banach algebra under pointwise function multiplication \cite{FeichtingerEtAl07}.  

In this paper, we study harmonic Hilbert spaces induced by \emph{subconvolutive} functions $\lambda \in L^1(\hat G)$, satisfying
\begin{equation}
    \label{eqSubconv0}
    (\lambda * \lambda) (\gamma) \leq C \lambda (\gamma), \quad \gamma \in \hat G,
\end{equation}
along with strict positivity and symmetry conditions (see \cite{Feichtinger79}*{Satz 3.6} and related ideas in earlier papers such as \cites{Edwards59,Essen73}). A prototypical example is the class of subexponential weights on the dual group $ \hat G = \mathbb Z^d$ of $G = \mathbb T^d $,
\begin{equation}
    \label{eqSubexp}
    \lambda^{-1}(\gamma) = e^{\tau \lvert \gamma \rvert^p}, \quad \tau > 0, \quad p \in (0,1).
\end{equation}
One of our main results, Theorem~\ref{thmMain}, is that on a compact abelian group $G$ every such space $\mathcal H_\lambda$ is a symmetric Banach $^*$-algebra under pointwise function multiplication and complex conjugation. In addition, in Theorem~\ref{thmSpec} we show that $\mathcal H_\lambda$ has the same spectrum (space of maximal ideals) as the $C^*$-algebra of continuous functions on $G$. We also study one-parameter families of RKHSs associated with self-adjoint Markov operators on $L^2(G)$. In this setting, we find that subconvolutivity of $\lambda$ is a necessary and sufficient condition for these spaces to have Banach algebra structure under pointwise function multiplication; see Theorem~\ref{thmMarkov}. 

We will refer to RKHSs which are simultaneously Banach $^*$-algebras with respect to pointwise function multiplication and complex conjugation as \emph{reproducing kernel Hilbert algebras} (RKHAs).  

As an application of our results, we discuss connections between RKHAs and a class of Banach spaces $\mathcal A_w$ on a compact abelian group $G$ induced from weights $w: \hat G \to \mathbb R_+$ such that $\mathcal A_w$ is a subspace of the Wiener algebra $\mathcal A(G)$ of continuous functions with absolutely convergent Fourier series. Spaces in this class have recently received attention in the context of function approximation in high-dimensional periodic domains, $G = \mathbb T^d$ with $d \gg 1$, where the weight function $w \equiv w_{s,r}$ is oftentimes from the class of dominating mixed smoothness, 
\begin{equation}
    \label{eqWSR}
    w_{s,r}(\gamma) = \prod_{i=1}^d (1 + \lvert \gamma_i \rvert^r)^{s/r}, \quad r,s \in (0, \infty),
\end{equation}
with $\gamma = (\gamma_1, \ldots, \gamma_d) \in \hat G = \mathbb Z^d$ and a related definition for $r = \infty $ \cites{KolomoitsevEtAl21,NguyenEtAl22}. Kolomoitsev, Lomako, and Tikhonov \cite{KolomoitsevEtAl21} refer to the spaces $\mathcal A_w(\mathbb T^d)$ as weighted Wiener spaces. In the paper \cite{NguyenEtAl22}, Nguyen, Nguen, and Sickel call $\mathcal A_w(\mathbb T^d)$ weighted Wiener algebras, though they do not claim that these spaces are actually algebras under pointwise multiplication.  In Section~\ref{secWienerAlg}, we consider cases where $\mathcal A_w$ and $\mathcal H_\lambda$ with $\lambda = w^{-2}$ are both Banach algebras associated with subconvolutive weights (which include the $w_{s,r}$ family for $s \geq 2$), using results from \cites{Brandenburg1975,Kuznetsova06} to verify the Banach algebra structure of $\mathcal A_w$. Following a suggestion of Feichtinger \cite{Feichtinger22}, we call the Banach algebras $\mathcal A_w$ Fourier--Wermer algebras owing to the fact that they are Fourier images of convolution algebras $L^1_w(\hat G)$ on the dual group $\hat G$ that were studied in the early paper of Wermer \cite{Wermer54}. We establish embedding relationships between $\mathcal A_w$ and $\mathcal H_\lambda$ that should be useful in function approximation applications.       

\subsection{\label{secRelated}Related work} 

The study of Banach algebras on LCAs associated with subconvolutive weights was initiated by Feichtinger \cite{Feichtinger79}. That work established that $L^\infty_w$ is a Banach convolution algebra iff $ w^{-1}$ is subconvolutive. It was also shown that if $w$ is subadditive and $w^{-1}$ lies in $L^1$, then $w^{-1}$ is subconvolutive. Thus, subadditive weights with integrable inverses provide a useful route to constructing subconvolutive weights and associated $L^\infty_w$ convolution algebras. In general, however, subadditivity and subconvolutivity are independent notions. Indeed, the subexponential weights $w=\lambda^{-1}$ from~\eqref{eqSubexp} are subconvolutive but not subadditive. Another relevant property of weight functions on LCAs is \emph{submultiplicativity}, 
\begin{displaymath}
    w(\gamma+\gamma') \leq C w(\gamma)w(\gamma'),
\end{displaymath}
the main result here being that $L^1_w$ is a Banach convolution algebra iff $w$ is submultiplicative \cite{Feichtinger79}. 

Besides $L^1_w$ and $L^\infty_w$, sufficient conditions for $L^p_w$, $ 1 < p < \infty $, to be Banach convolution algebras are also known. An early reference in that direction is the paper of Wermer \cite{Wermer54}, which studies $L^p_w$ convolution algebras on the real line, including the case $p=1$. More recently, Kuznetsova \cite{Kuznetsova06} has shown (using methods of proof for sequence spaces from~\cite{Nikolskii74}) that for a locally compact group, $L^p_w$ with $ 1 < p < \infty$ is a convolution algebra if
\begin{equation}
    \label{eqSubconvW}
    (w^{-p} * w^{-p})(\gamma) \leq C w^{-p}(\gamma).
\end{equation}
When $p=2$, the latter condition is equivalent to~\eqref{eqSubconv0} with $ \lambda = w^{-2}$. Related results for weighted convolution algebras on $\mathbb R^n$ have been obtained by Kerman and Sawyer \cite{KermanSawyer94}. 

See \cites{ReitSteg2000,Grochenig07} for surveys of weight functions in harmonic analysis. Closely related to weight function theory is the theory of Sobolev spaces and algebras on LCAs, e.g., \cites{FeichtingerWerther04,Gorka2014,BernicotFrey2018,BrunoEtAl2019,GorkaKostrzewa20}. 

In the paper \cite{Essen73}, Ess\'en studies the spectral properties of Banach convolution algebras on $\mathbb Z$ or $\mathbb R$ under subconvolutive weights with additional decay conditions. In \cite{Brandenburg1975}, Brandenburg establishes sufficient conditions for the equivalence of the spectrum of unital commutative Banach algebra $B$ and a subalgebra $S \subset B$. Using these conditions, is shown that for a subadditive weight function $w \geq 1$ on an LCA, the space $L^1 \cap L^\infty_w$ (equipped with the norm $\lVert f \rVert_{L^1} + \lVert f \rVert_{L^\infty_w}$) is a Banach convolution algebra which has the same spectrum as the group convolution algebra $L^1$. The spectral properties of $L^p_w$ convolution algebras on LCAs were studied in \cite{Kuznetsova06}, where it was shown that every such algebra is semisimple. In more recent work, Kuznetsova and Molitor-Braun \cite{KuznetsovaMolitorBraun12} studied the representation theory for convolution Banach $^*$-algebras on locally compact, non-abelian groups, and established sufficient conditions for these algebras to be symmetric (i.e., the spectra of positive elements are subsets of the positive half-line), among other results. 

A related notion of Banach algebras on locally compact groups with Hilbert space structure is the class of $H^*$-algebras proposed by Ambrose \cite{Ambrose45}. An $H^*$-algebra $H$ is required to satisfy the identity
\begin{equation}
    \label{eqHStar}
    \langle fg, h \rangle_H = \langle g, f^*h \rangle_H = \langle f, h g^* \rangle_H, \quad \forall g,h\in H,   
\end{equation}
which implies that for each $f\in H$, the operations of left and right multiplication by $f$ are norm-preserving, $^*$-homomorphisms between $H$ and $B(H)$ (the $C^*$-algebra of bounded linear maps on $H$). A classical example of an $H^*$-algebra is the $L^2$ convolution algebra on a compact group. Note that the RKHAs and weighted convolution algebras on LCAs mentioned above do not, in general, satisfy~\eqref{eqHStar}.

In this paper, we focus on the setting of \emph{compact} abelian groups, which allows us to approach the problem of constructing RKHAs using Mercer theory. In particular, in this setting the values $\lambda(\gamma)$ correspond to the eigenvalues of a compact integral operator on $L^2(G)$ associated with the reproducing kernel of $\mathcal H_\lambda$, and the characters $\gamma \in \hat G$, $ \gamma : G \to  S^1 \subset \mathbb C$ form a corresponding orthogonal eigenbasis. This structure allows us to deduce that $\mathcal H_\lambda$ is a Banach algebra if $\lambda \in L^1(\hat G)$ is subconvolutive, which is more general than the subconvolutivity implied by subadditivity of $ \lambda^{-1} = w^2$. Mercer theory also facilitates the characterization of the associated maximal ideal spaces using kernel integral operators. 

\subsection{Plan of the paper} 

We introduce our notation in Section~\ref{secPrelim}, and give an overview of relevant results from RKHS theory in Section~\ref{secRKHS}. In Section~\ref{secSubconv}, we describe the construction of RKHAs on compact abelian groups associated with subconvolutive functions. In Section~\ref{sec:spec}, we study the spectra and state spaces of RKHAs (not necessarily associated with subconvolutive functions). In Section~\ref{secMarkov} we study one-parameter families of RKHAs associated with Markov semigroups. Section~\ref{secWienerAlg} discusses aspects of Fourier--Wermer algebras associated with subconvolutive weights and their embedding relationships with RKHAs. Appendix~\ref{sec:proof:LCA} contains the proof of an auxiliary result, Lemma~\ref{lem:LCA}, on integral operators associated with translation-invariant kernels on LCAs.      

\section{\label{secPrelim}Notation and preliminaries}

\subsection{Locally compact abelian groups}
Consider an LCA $G$ equipped with a Haar measure $\mu$. In most instances, we will assume that $G$ is compact, in which case $\mu$ will be normalized to a probability measure. We let $\hat{G} $ denote the dual group of $G$, i.e.,  the abelian group of continuous homomorphisms $\gamma : G \to S^1$, equipped with its dual measure, $ \hat \mu $ \cites{Rudin1962,Morris1977}. We identify each element of $ \hat{G} $ with a continuous, complex-valued function on $G$, taking values in the unit circle $S^1 \subset \mathbb C $, and acting on $ \mathbb C$ multiplicatively as a unitary character. The trivial character in $ \hat G $ will be denoted by $0_{\hat G}$. When there is no risk of confusion with scalar multiplication of functions, the inverse of $ \gamma \in \hat G $ will be denoted by $-\gamma$. We recall that if $G$ is compact, $\hat G $ has a discrete topology and $\hat \mu $ is a weighted counting measure. 

In what follows, $C_0(G)$ (resp.\ $C_0( \hat{G})$) will denote the Banach space of complex-valued, continuous functions on an LCA $G$ (resp.\ its dual $ \hat{G}$) vanishing at infinity, equipped with the uniform norm. Moreover, $\Fourier : L^1(G) \to C_0(\hat{G})$ and $\Inverse : L^1(\hat{G})\to C_0(G)$ will denote the Fourier and inverse Fourier transforms, respectively,  i.e.,
\[
 \Fourier f(\gamma) := \int_G f(x) \gamma(-x) \, d\mu(x), \quad \Inverse  \hat f( x ) := \int_{\hat{G} } \hat f(\gamma ) \gamma(x) \, d\hat \mu(\gamma).
\]
We also use $*$ and $^\star$ to denote the convolution (algebraic product) and antilinear involution operations on the group algebra $L^1(G)$, respectively, i.e., 
\[( f * g )(x) = \int_G f(x - y) g(y) \, d\mu(y), \quad f^\star(x) = \overline{f(-x)}, \]
and a similar notation for the corresponding operations on $L^1(\hat{G})$. When working with Hilbert spaces such as $L^2(G)$ or $\mathcal H_\lambda$, we will adopt the convention that the inner product is antilinear in the first argument, e.g., $\langle f, g \rangle_{L^2(G)} = \int_G \bar f g \, d\mu$.

Throughout the paper, $S^x $ will denote the (left) shift operator by group element $ x \in G$, i.e., $ S^x f( y ) = f( x + y ) $ for any element $ y \in G $ and function $ f : G \to \mathbb C$. The collection $ \{ S^x \}_{x \in G}$ forms a strongly continuous group of isometries on any of the spaces $L^p(G)$ with $1 \leq p < \infty$ and $C_0(G)$. We also recall the standard property of Fourier transforms that 
\begin{equation}\label{eqn:FourierShift}
\Fourier(S^x f)(\gamma) = \gamma(x) ( \Fourier f )( \gamma ), \quad \forall f \in L^1(G), \quad \forall x \in G, \quad \forall \gamma \in \hat{G}. 
\end{equation}

Reusing notation, we shall let $ \mathcal F : L^2(G) \to L^2( \hat G)$ denote the unitary extension of the Fourier operator. Similarly, we will continue to use $*$ to denote the convolution operator on $L^2(G)$. Note that for any $f,g \in L^2(G)$, $ f * g $ lies in $C_0(G)$ \cite{Folland95}*{Proposition~2.40}, and we have 
\begin{equation}
    \label{eqConvShift}
    ( f * g )(x) = \langle f^\star, S^x g \rangle_{L^2(G)}. 
\end{equation}
We will let $\hat S^\gamma$ be the shift operator on the dual group $\hat G$, defined analogously to $S^x$ on $G$ and satisfying the corresponding relations in~\eqref{eqn:FourierShift} and~\eqref{eqConvShift}. 

We say that a function $f:G\to \cmplx$ on an LCA is \emph{uniformly continuous} if for every $\epsilon>0$ there is a neighborhood $U$ of the identity element of $G$, such that for every $x\in G$ and $ y\in  x+U$, $ \abs{ f(x) - f(y) } <\epsilon $. This notion is equivalent to uniform continuity of functions on uniform spaces \cite{James_uniform_2012}*{Definition~7.6}.

\subsection{Reproducing kernel Hilbert algebras}

Let $\lambda$ be a positive-valued function in $L^1(\hat G)$. In this paper, our focus is on RKHSs on the group $G$ with translation-invariant kernels $k: G \times G \to \mathbb C$ induced by $\lambda$, viz.
\begin{equation} \label{eqK}
    k(x,y) = l( x - y ), \quad l :=  \hat{ \Fourier } \lambda \in C_0(G).
\end{equation}
By Bochner's theorem for LCAs \cite{Rudin1962}*{Section~1.4.2}, $k $ is the reproducing kernel of an RKHS $ \mathcal H_\lambda $ of continuous functions; that is, 
\begin{displaymath}
    f(x) = \langle k(x,\cdot), f \rangle_{\mathcal H_\lambda}, \quad \forall f \in \mathcal H_\lambda, \quad \forall x \in G,
\end{displaymath}
which expresses the fact that pointwise evaluation functionals on $\mathcal H_\lambda$ are continuous. Note that $\mathcal H_\lambda$ is a harmonic Hilbert space associated with the weight function $w = \lambda^{-1/2}$. 

We will say that $\mathcal H_\lambda$ is a \emph{reproducing kernel Hilbert algebra} (RKHA) if it is a Banach $^*$-algebra with respect to pointwise function multiplication and complex conjugation, i.e., 
\begin{equation}
    \label{eqRKHA}
    \lVert fg\rVert_{\mathcal H_\lambda} \leq C \lVert f\rVert_{\mathcal H_\lambda} \lVert g\rVert_{\mathcal H_\lambda}, \quad \lVert f^* \rVert_{\mathcal H_\lambda} \equiv  \lVert \bar f \rVert_{\mathcal H_\lambda} = \lVert f \rVert_{\mathcal H_\lambda}, \quad \forall f, g \in \mathcal H_\lambda.
\end{equation}
If $\mathcal H_\lambda$ is unital, we will use the symbol $1_G$ to denote the unit of $ \mathcal H_\lambda$, i.e., the function equal to 1 at every point in $G$. Note that we do not require that the norm of $1_G$ is equal to 1. We will also let $\sigma_\lambda(f)$ denote the spectrum of $f \in \mathcal H_\lambda$, i.e., the set of complex numbers $z$ such that $ f - z $ does not have a multiplicative inverse in $\mathcal H_\lambda$. We recall that $\mathcal H_\lambda$ is \emph{symmetric} as a unital Banach $^*$-algebra if $\sigma_\lambda(f^*f)$ is a subset of $[0, \infty) $ for every element $f\in \mathcal H_\lambda$.   

\begin{remark}
    \label{remC}We have stated the Banach algebra condition $ \lVert fg \rVert_{\mathcal H_\lambda} \leq C \lVert f \rVert_{\mathcal H_\lambda} \lVert g \rVert_{\mathcal H_\lambda}$  allowing a general constant $C$, as opposed to the more conventional definition $ \lVert fg \rVert_{\mathcal H_\lambda} \leq \lVert f \rVert_{\mathcal H_\lambda} \lVert g \rVert_{\mathcal H_\lambda}$. This choice does not affect any of the results presented below, as $C $ can be absorbed in a redefinition of the reproducing kernel $k $ of $\mathcal H_\lambda$ to a scaled kernel $ \tilde k := C^2 k $. The corresponding RKHA, $\tilde{\mathcal H}_\lambda$, has the same elements as $\mathcal H_\lambda$, and satisfies $ \lVert fg \rVert_{\tilde{\mathcal H}_\lambda} \leq \lVert f \rVert_{\tilde{\mathcal H}_\lambda} \lVert g \rVert_{\tilde{\mathcal H}_\lambda}$, so we can view $ \lVert \cdot \rVert_{\tilde{\mathcal H}_\lambda}$ as an equivalent norm to $ \lVert \cdot \rVert_{\mathcal H_\lambda}$. That being said, it should be kept in mind that a number of standard results on Banach algebras require appropriate modification when $C > 1 $. For example, the fact that a state on a unital Banach $^*$-algebra with isometric involution has unit operator norm does not necessarily hold when $ C > 1$. In the present work, it is natural to allow a general $C$, as this enables a direct identification of the reproducing kernels of certain RKHAs with Markov transition kernels without having to employ additional normalization factors. 
\end{remark}

\section{\label{secRKHS}Results from reproducing kernel Hilbert space theory}

In this section, we collect results from RKHS theory that will be useful in the analysis that follows.   

First, consider a locally compact Hausdorff space $X$. We use $ \mathcal P(X)$ to denote the set of Borel probability measures on $X$ and $ \mathbb E_\nu(\cdot) = \int_X (\cdot) \, d\nu$ the expectation operator with respect to $\nu \in \mathcal P(X)$.

We recall that a positive-definite kernel $k$ on $X$ whose corresponding RKHS $\mathcal H $ is dense in $C_0(X)$ is known as \emph{$C_0$-universal} \cite{CarmeliEtAl10}, or \emph{$C$-universal} if $X$ is compact \cite{Steinwart_inflnc_2001}. If the map $ R : \mathcal P( X ) \to \mathcal H$ with 
\begin{equation}
    \label{eqKME} R(\nu) = \int_X k(x,\cdot)\, d\nu(x) 
\end{equation}
is well-defined and injective, the kernel $k$ is called \emph{characteristic} \cite{FukumizuEtAl07}. In that case, $R$ is referred to as a \emph{kernel mean embedding} of probability measures \cite{SmolaEtAl07}. Moreover, for any $f\in \mathcal H $ and $\nu \in \mathcal P(X)$, we have
\begin{displaymath}
    \mathbb E_\nu f = \langle R(\nu), f \rangle_{\mathcal H}.
\end{displaymath}
Thus, we can evaluate expectation values of elements of $\mathcal H$ by means of Hilbert space inner products. For a characteristic kernel, the \emph{feature map} $ F : X \to \mathcal H$ with $ F(x) = k(x,\cdot) $ is injective, and has linearly independent range \cite{Steinwart_inflnc_2001}.  

On a compact Hausdorff space $X$, a $C$-universal kernel $k$ is strictly positive-definite and characteristic \cites{GrettonEtAl07,SriperumbudurEtAl2011}. Moreover, the kernel mean embedding induced by $k$ metrizes the weak-$^*$ topology of $\mathcal P(X)$ \cite{SriperumbudurEtAl2010}. That is, a sequence of measures $\nu_j \in \mathcal P(X)$ converges to $\nu \in \mathcal P(X)$ weak-$^*$ sense iff $R(\nu_j)$ converges to $R(\nu)$ in the norm of $\mathcal H$. See \cite{FukumizuEtAl08} for a study of characteristic kernels on LCAs.

The following are standard results from Mercer theory \cites{FerreiraMenegatto2013, SriperumbudurEtAl2011, PaulsenRaghupathi2016}, which we state without proof.

\begin{lemma} \label{lem:Mercer}
Let $X$ be a compact Hausdorff space, $\mu $ a finite Borel measure with full support in $X$, and $ k : X \times X \to \mathbb C$ a positive-definite, continuous kernel with associated RKHS $\RKHS$. Then, the following hold.
\begin{enumerate}[(i)]
        \item $\RKHS$ is a subspace of $C(X) $, and the inclusion $ \RKHS \hookrightarrow C(X)$ is compact.
        \item $ K : L^2(X) \to \RKHS$ with $ K f = \int_X k( \cdot, x ) f(x) \, d\mu(x) $ is a well-defined, compact integral operator with dense range. 
        \item The adjoint $K^* : \RKHS \to L^2(X)$ is equal to the restriction of the inclusion map $ C(X) \hookrightarrow L^2(X) $ on $ \RKHS$; that is, $ K^* f = f $, $ \mu$-a.e.
        \item $ \mathcal K := K^* K $ is a positive, self-adjoint Hilbert-Schmidt operator on $L^2(X)$, with eigenvalues $ \lambda_0 \geq \lambda_1 \geq \cdots \searrow 0$ and a corresponding orthonormal basis $ \{ \phi_0, \phi_1, \ldots \} $ of eigenfunctions. Moreover, if $ k$ is positive-valued, $ \mathcal K$ is of trace class. 
        \item The set $ \{ \psi_j = \lambda_j^{-1/2 } K \phi_j : \lambda_j > 0 \} $ is an orthonormal basis of $\RKHS$, satisfying $K^* \psi_j = \lambda_j^{1/2} \phi_j$. 
        \item The kernel admits the Mercer series expansion
            \begin{displaymath}
                k(x,y) = \sum_{j : \lambda_j >0 } \overline{\psi_j(x)} \psi_j(y),
            \end{displaymath}
            which converges uniformly for $ (x,y) \in X \times X$.
        \item If the $\lambda_j$ are strictly positive, $k$ is $C$-universal (and thus strictly positive-definite and characteristic).  
    \end{enumerate}
\end{lemma}

Next, we state certain properties of integral operators associated with translation-invariant kernels on LCAs.

\begin{lemma}\label{lem:LCA}
    Let $\lambda \in L^1(\hat G ) \cap C_0(\hat G)$ be an absolutely integrable, continuous positive-valued function on the dual group $\hat G$ of an LCA $G$ with absolutely integrable Fourier transform $l = \hat{\mathcal F}\lambda \in L^1(G)$, and let $k : G \times G \to \mathbb C$ be the corresponding translation-invariant reproducing kernel from~\eqref{eqK}. Then:
    \begin{enumerate}[(i)] 
        \item $k$ is uniformly continuous, and $ \RKHS_\lambda $ is a subspace of $C_0(G)$. 
        \item The integral operator 
            \begin{displaymath}
                K : f \mapsto \int_G k( \cdot , x ) f(x) \, d\mu(x) 
            \end{displaymath}
            maps $L^\infty(G)$ into the Banach space of bounded functions on $G$. Moreover, $ K$ maps $L^2(G)$ and $L^1(G)$ into the space of uniformly continuous functions on $G$.
        \item For every $x \in G$ and  $\gamma \in \hat{G}$, 
	\begin{displaymath} 
        K \gamma(x) = \lambda(\gamma) \gamma(x).
\end{displaymath}
    \end{enumerate}
\end{lemma}

\begin{proof}
    See Appendix~\ref{sec:proof:LCA}.\phantom\qedhere
\end{proof}

Recall that if $G$ is compact, then $\hat{G}$ has a discrete topology. By Lemmas~\ref{lem:Mercer} and~\ref{lem:LCA}(iii), when $G$ is compact we can identify the eigenfunctions $ \phi_j $ of the (compact) integral operator $ \mathcal K : L^2(G) \to L^2(G)$ with the characters of $G$, $\phi_j \equiv \phi_\gamma \equiv \gamma$. Using $ \Lambda \subseteq \hat G$ to denote the set $\{ \gamma \in \hat{G}: \lambda(\gamma) > 0 \}$, and defining $ \xi(\gamma) := \sqrt{\lambda(\gamma)}$, the corresponding basis functions $ \psi_j \equiv \psi_\gamma$ of $\mathcal H_\lambda$ are
\begin{equation}\label{eqn:def:psi}
    \psi_{\gamma} := \frac{1}{ \xi(\gamma) } K \gamma = \xi(\gamma) \gamma, \quad \gamma \in \Lambda.
\end{equation}
As a result, the elements of $\RKHS_\lambda$ can be explicitly characterized as
\[ \RKHS_\lambda = \left\{  f = \sum_{\gamma\in \Lambda} \hat f_\gamma \gamma = \sum_{\gamma\in\Lambda} \hat f_\gamma \psi_{\gamma}/ \xi(\gamma): \; \sum_{\gamma\in \Lambda} \lvert \hat f_\gamma \rvert^2 / \lambda(\gamma) < \infty \right\}. \]

In the above, the coefficients $ \hat f_\gamma$ coincide with the values of the Fourier transform of the continuous function $ f \in \RKHS_\lambda \subseteq C(G) $ on $ \Lambda$, i.e., $ \hat f_\gamma = \mathcal F f(\gamma)$. Moreover, the condition $\sum_{\gamma\in\Lambda} \lvert \hat f_\gamma \rvert^2 / \lambda(\gamma) < \infty$ is equivalent to the statement that the function $ \hat u : \hat G \to \mathbb C$ with 
\begin{displaymath}
    \hat u( \gamma ) = 
    \begin{cases}
        \hat f_\gamma / \xi(\gamma), & \gamma \in \Lambda,\\
        0, & \text{otherwise},
    \end{cases}
\end{displaymath}
lies in $L^2(\hat G) $. Together, these facts imply:

\begin{lemma} \label{lem:U}
	The following statements are equivalent:
    \begin{enumerate}[(i)]
        \item $ f $ is an element of $ \RKHS_\lambda $. 
        \item There exists $ \hat u \in L^2(\hat G)$ with $ \lVert f \rVert_{\mathcal H_\lambda} = \lVert \hat u \rVert_{L^2(\hat G)}$  such that $ \mathcal F f = \xi \hat u$.
    \end{enumerate}
    Moreover, $ \hat u $ is unique, and can be explicitly constructed as $ \hat u = \xi^+ \mathcal F f $, where 
    \begin{displaymath}
    \xi^+(\gamma) = 
    \begin{cases}
        1/ \xi(\gamma), & \gamma \in \Lambda, \\
        0, & \gamma \in \hat G \setminus \Lambda.
    \end{cases}
\end{displaymath}
\end{lemma}

\section{\label{secSubconv}Reproducing kernel Hilbert algebras from subconvolutive functions}

Unless otherwise stated, throughout this section we will assume that $G$ is compact. One of our main results is the following.

\begin{theorem}
    \label{thmMain}
    Suppose that $\lambda \in L^1(\hat G)$ is: 
    \begin{enumerate}[(i)]
        \item Strictly positive-valued, $\lambda(\gamma) > 0 $;
        \item Subconvolutive, $( \lambda * \lambda )(\gamma ) \leq C \lambda(\gamma)$;
        \item Symmetric, $\lambda(-\gamma) = \lambda(\gamma)$. 
    \end{enumerate}
    Then, $\mathcal H_\lambda$ is a unital, symmetric Banach $^*$-algebra with respect to pointwise multiplication and complex conjugation of functions, and lies dense in $C(G)$. 
\end{theorem}

\begin{proof}
    See Section~\ref{secProofMain}.\phantom\qedhere
\end{proof}

\begin{remark}
    The fact that  under Assumption~(ii) of Theorem~\ref{thmMain} $\mathcal H_\lambda $ is a Banach algebra can be readily deduced from \cite{Kuznetsova06}*{Theorem~1}, which came to our attention after completion of this work. In particular, \cite{Kuznetsova06}*{Theorem~1} establishes that under Assumption~(ii) $L^2_w(\hat G)$ with $ w=\lambda^{-1/2} $ is a convolution algebra, and taking Fourier transforms yields the Banach algebra property of $\mathcal H_\lambda$ under pointwise multiplication. See also \cites{Nikolskii74,KermanSawyer94,KuznetsovaMolitorBraun12} for related work. In Section~\ref{secProofMain}, we include a proof tailored to the Hilbert space setting, which uses Cauchy-Schwarz inequalities and the representation of convolution as an $L^2(\hat G)$ inner product (rather than, e.g.,  H\"older inequalities employed in \cite{KuznetsovaMolitorBraun12}).
\end{remark}

As mentioned in Section~\ref{secRelated}, a useful way of constructing subconvolutive functions on $\hat G$ through positive-valued functions $ \lambda \in L^1(\hat G)$ with subadditive inverses,
\begin{displaymath}
    \lambda^{-1}(\gamma + \gamma') \leq C ( \lambda^{-1}(\gamma) + \lambda^{-1}(\gamma')).  
\end{displaymath}
By \cite{Feichtinger79}*{Corollary~3.8}, every such function $\lambda$ is subconvolutive. If, in addition, $\lambda$ is strictly positive-valued and symmetric, then by Theorem~\ref{thmMain} $\mathcal H_\lambda$ is an RKHA. The subexponential weights $\lambda^{-1}$ in~\eqref{eqSubexp} also satisfy the assumptions of Theorem~\ref{thmMain}, but in this case $\lambda^{-1}$ is not subadditive. Another example mentioned in Section~\ref{secRelated} is that of an LCA $G$ (not necessarily compact) with a function $\lambda \in L^1(\hat G)$ such that $\lambda^{-1/2}$ is subadditive. In that case, too, $\mathcal H_\lambda$ is a Banach algebra with respect to pointwise multiplication \cite{FeichtingerEtAl07}. To make contact with that result in the compact case, we note the following fact.  

\begin{lemma}
    Suppose that $G$ is compact and $ \xi \in L^2(\hat G ) $ is positive-valued, self-adjoint with respect to convolution ($\xi^\star = \xi)$, and subconvolutive. Then, $ \lambda = \xi^2 \in L^1(\hat G)$ is also subconvolutive. 
    \label{lemSubconv}
\end{lemma}

\begin{proof}
    Since $G $ is compact, we have $L^1(\hat G ) \subseteq L^2(\hat G)$ and $\lVert \cdot \rVert_{L^2(\hat G)} \leq \lVert \cdot \rVert_{L^1(\hat G)}$ (since the dual measure $\hat \mu $ is a normalized counting measure). Thus, using~\eqref{eqConvShift} and the facts that $\xi^\star= \xi$ and $\xi * \xi \leq C \xi$, we get
    \begin{align*}
        \lambda * \lambda(\gamma) &= \langle \lambda, S^\gamma \lambda \rangle_{L^2(\hat G)} = \langle \xi^2, S^\gamma \xi^2 \rangle_{L^2(\hat G)} = \langle \xi S^\gamma \xi, \xi S^\gamma \xi \rangle_{L^2(\hat G)} = \lVert \xi S^\gamma \xi \rVert_{L^2(\hat G)}^2 \\
        &\leq \lVert \xi S^\gamma \xi \rVert_{L^1(\hat G)}^2  = \langle \xi, S^\gamma \xi \rangle_{L^2(\hat G)}^2 = ( \xi * \xi(\gamma) )^2 \\
        &\leq C^2 \xi^2(\gamma) = C^2 \lambda(\gamma). \qedhere
    \end{align*}
\end{proof}

If $G$ is compact and $\xi \in L^1(\hat G) \subseteq L^2(\hat G)$ has a subadditive inverse, then by \cite{Feichtinger79}*{Corollary~3.9}, $ \xi$ is subconvolutive. As a result, by Lemma~\ref{lemSubconv}, $ \lambda = \xi^2 $ is also subconvolutive.  Thus, under the additional constraint $\lambda^{1/2} \in L^1(\hat G)$, the subadditivity assumption on $\lambda^{-1/2}$ underlying the construction of the Banach algebra $\mathcal H_\lambda$ in \cite{FeichtingerEtAl07} implies the subconvolutivity assumption in Theorem~\ref{thmMain}. Note that the subexponential weights in~\eqref{eqSubexp} satisfy $\lambda^{1/2} \in L^1(\hat G)$.

\subsection{\label{secProofMain}Proof of Theorem~\ref{thmMain}}

Consider two elements $f,g \in \RKHS_\lambda$. By Lemma~\ref{lem:U}, to show that the continuous function $fg$ lies in $ \RKHS_\lambda$ it is enough to show that the function $ \hat w : \hat G \to \mathbb C$ defined as $ \hat w = \xi^+ \mathcal F(fg)$ lies in $L^2(\hat G)$. To that end, letting $ \hat u = \xi^+ \mathcal F f $ and $ \hat v = \xi^+ \mathcal F g $ be the $L^2(\hat G)$ representatives of $ f $ and $ g$ from Lemma~\ref{lem:U}, we obtain
\begin{displaymath}
    \mathcal F(fg)(\gamma) = ( \mathcal F f * \mathcal F g )(\gamma) = ( ( \xi \hat u ) *  ( \xi \hat v ) )( \gamma ) = \langle ( \xi \hat u )^{\star}, \hat S^\gamma( \xi \hat v) \rangle_{L^2(\hat G)}.
\end{displaymath}
Then, using standard properties of shift operators and $L^2$ inner products, as well as the fact that $ \xi $ is real and $L^1(\hat G)$-self-adjoint, we get
\begin{align*}
    \lvert \mathcal F(fg)(\gamma) \rvert^2 &= \lvert \langle (  \xi \hat u )^{\star}, \hat S^\gamma( \xi \hat v) \rangle_{L^2(\hat G)} \rvert^2 = \lvert \langle \xi \hat u^\star, ( \hat S^\gamma \xi) ( \hat S^\gamma \hat  v) \rangle_{L^2(\hat G)} \rvert^2 \\
    &= \lvert \langle \xi S^\gamma \xi, \overline{\hat u^*} \hat S^\gamma \hat v \rangle_{L^2(\hat G)} \rvert^2 \leq  \langle \xi S^\gamma \xi,  \xi S^\gamma \xi \rangle_{L^2(\hat G)} \langle\overline{\hat u^*} \hat S^\gamma v, \overline{\hat u^*} \hat S^\gamma v \rangle_{L^2(\hat G)} \\
    &= \langle \lambda, \hat S^\gamma \lambda \rangle_{L^2(\hat G)} \langle \lvert \hat u^\star \rvert^2, S^\gamma \lvert \hat v \rvert^2 \rangle_{L^2(\hat G)} = [ ( \lambda* \lambda)(\gamma) ] [ ( \lvert \hat u^\star \rvert^2 * \lvert \hat v \rvert^2 )( \gamma) ] \\
    & \leq C \lambda(\gamma) [ ( \lvert \hat u^\star \rvert^2 * \lvert \hat v \rvert^2 )( \gamma) ],
\end{align*}
where we used the subconvolutivity of $\lambda$ to arrive at the last line. Thus, since $ \lambda$ is strictly positive-valued, we have $ \xi^+(\gamma) = 1 /  \sqrt{\lambda}(\gamma)$, and $ \lvert \hat w(\gamma) \rvert^2 \leq C ( \lvert \hat u^\star \rvert^2 * \lvert \hat v \rvert^2 )( \gamma) $. Therefore, 
\begin{displaymath}
    \lVert \hat w \rVert_{L^2(\hat G)}^2 \leq C \lVert \lvert \hat u^\star \rvert^2 * \lvert \hat v \rvert^2 \rVert_{L^1(\hat G)}, 
\end{displaymath}
and it follows that $\hat  w$ lies in $L^2(\hat G)$ since $\lvert \hat u^\star \rvert^2 * \lvert \hat v \rvert^2 $ is the convolution of the $L^1(\hat G)$ elements $ \lvert \hat u^\star \rvert^2 $ and $ \lvert \hat v \rvert^2 $. 

We thus conclude that $\mathcal H_\lambda$ is a Banach algebra with respect to pointwise function multiplication. The fact that $\mathcal H_\lambda $ is a dense subspace of $C(G)$ follows from the strict positivity of $\lambda$ in conjunction with Lemma~\ref{lem:Mercer}(vii). 

Next, we verify that $f^* (x) = \overline{f(x)}$ is an isometric, antilinear involution on $\RKHS_\lambda$. Since $ \lambda(\gamma) = \lambda(-\gamma)$ for every $ \gamma \in \hat G $, the orthonormal basis elements $\psi_{\gamma}$ from~\eqref{eqn:def:psi} satisfy 
\begin{displaymath}
    \overline{ \psi_{\gamma}(x) } = \overline{\lambda^{1/2}(\gamma) \gamma(x)} = \lambda^{1/2}(-\gamma) \overline{\gamma(x)} = \lambda^{1/2}(-\gamma) \gamma^{-1}(x) = \psi_{-\gamma}(x),
\end{displaymath}
so $ \lVert \psi^*_{\gamma} \rVert = \lVert \psi_{-\gamma} \rVert_{\mathcal H_\lambda} = 1$. Therefore, $^*$ preserves the norm of orthonormal basis vectors of $ \RKHS_\lambda$. Moreover, it is clearly antilinear and involutive, so $\mathcal H_\lambda$ is a Banach $^*$-algebra satisfying~\eqref{eqRKHA}. 

The RKHA $\RKHS_\lambda$ is also unital and satisfies $ \lVert 1_G \rVert_{\mathcal H_\lambda} = 1$ since the unit basis vector $\psi_0$ is equal to the trivial character in $ \hat G $, and thus everywhere equal to 1 on $ G$. 

Finally, the symmetry of $\mathcal H_\lambda$ follows from Corollary~\ref{corSpec2} in Section~\ref{sec:spec} below. \qed

\section{Spectra and states of reproducing kernel Hilbert algebras} \label{sec:spec}

In general, an RKHA $ \mathcal H_\lambda$ on a compact abelian group $G$ does not satisfy the $C^*$ identity, $ \lVert f^* f \rVert_{C(G)} = \lVert f \rVert^2_{C(G)}$, holding for the $C^*$-algebra of continuous functions on $G$ under pointwise multiplication and complex conjugation, nor does it satisfy the $H^*$-identity in~\eqref{eqHStar} enjoyed by the $L^2(G)$ convolution algebra. Failure to meet, in particular, the last property means that the regular representation of $  \mathcal H_\lambda $ into $  B(\mathcal H_\lambda)$ is not a $^*$-representation. 

Yet, by virtue of their RKHS structure, RKHAs possess continuous evaluation functionals $ \delta_x : \mathcal H_\lambda \to \mathbb C$ at every $x \in G $, 
\begin{equation}
    \label{eqRep}
    \delta_x f = f(x) = \langle k( x, \cdot ), f \rangle_{\mathcal H_\lambda}, \quad \lVert \delta_x \rVert_{\mathcal H'_\lambda} = \sqrt{k(x,x)},
\end{equation}
satisfying 
\begin{displaymath}
    \delta_x(fg) = (\delta_x f ) ( \delta_x g ), \quad \delta_x f^* = \overline{\delta_x f}, \quad \forall f, g \in \mathcal H_\lambda, 
\end{displaymath}
where $ \lVert \cdot \rVert_{\mathcal H'_\lambda}$ is the operator norm of functionals in the dual space $ \mathcal H'_\lambda$. Every nonzero evaluation functional $\delta_x $ is an element of the spectrum of $\mathcal H_\lambda$, i.e., the set of nonzero homomorphisms of $\mathcal H_\lambda$ into $ \mathbb C $, denoted by $ \sigma( \mathcal H_\lambda)$. In addition, as we will see below, under appropriate conditions on the kernel, the $\delta_x $ provide an abundance of states on $\mathcal H_\lambda$, and also induce a set of states on the non-abelian $C^*$-algebra $B(\mathcal H_\lambda)$.

Recall now that for a compact Hausdorff space $G$, the spectrum of the $C^*$-algebra $C(G)$ consists precisely of the evaluation functionals $ \delta_x $ at every $ x \in G $ \cite{Dixmier1977}. Moreover, the map $ \beta : G \to \sigma(C(G))$ with $ \beta(x) = \delta_x $ and the Gelfand transform $ \Gamma : C(G) \to C(\sigma(C(G))) $ with $ ( \Gamma f )(\delta_x) = f(x)$ are homeomorphisms with respect to the weak-$^*$ topology of $\sigma(C(G))$. 

The following theorem characterizes analogously the spectra of RKHAs on compact abelian groups and the associated Gelfand transforms. 

\begin{theorem} \label{thmSpec}
    Let $\mathcal H_\lambda$ be an RKHA on a compact abelian group associated with a strictly positive function $\lambda \in L^1(\hat G)$. Then, the following hold.
    \begin{enumerate}[(i)]
        \item The map $ \beta_\lambda : G \to \sigma(\mathcal H_\lambda)$ with $ \beta_\lambda(x) = \delta_x$ is a homeomorphism  with respect to the weak-$^*$ topology on $\sigma(\RKHS_\lambda)$ inherited as a subset of $\RKHS'_\lambda$.
        \item Under the identification $ G \simeq \sigma(C(G))$ induced by $\beta$, the Gelfand transform $ \Gamma_\lambda : \mathcal H_\lambda \to C(\sigma(\mathcal H_\lambda))$ with $ ( \Gamma_\lambda f )(\delta_x) = f(x) $ coincides with the inclusion map $ \iota : \mathcal H_\lambda \hookrightarrow C(G) $. In particular, the operator norm of $ \Gamma_\lambda$ is equal to $ \sqrt{l(0_G)}$. 
    \end{enumerate} 
\end{theorem}

\begin{proof}
    See Section~\ref{secThmSpec}.\phantom\qedhere
\end{proof}

Theorem~\ref{thmSpec} establishes that $\mathcal H_\lambda$ has the same spectrum as $C(G)$. Analogous results were found in the paper~\cite{Brandenburg1975} for convolution algebras on LCAs. In particular, a class of convolution algebras associated with subadditive weight functions was shown to have the same spectrum as the group convolution algebra $L^1(G)$. Theorem~\ref{thmSpec} addresses the case of algebras with respect to pointwise function multiplication that are simultaneously RKHSs. In particular, our method of proof in Section~\ref{secThmSpec} makes explicit use of the RKHSs structure of $\mathcal H_\lambda$.   

The following are corollaries of Theorem~\ref{thmSpec}.

\begin{corollary}\ 
    \begin{enumerate}[(i)]
        \item Every non-vanishing function $f \in \mathcal H_\lambda$ has a multiplicative inverse in $\mathcal H_\lambda$.
        \item Every strictly positive function $ f \in \mathcal H_\lambda$ has a square root in $\mathcal H_\lambda$, i.e., there exists a (strictly positive) $ g \in \mathcal H_\lambda$ such that $ f = g^2 $.
        \item The spectrum $\sigma_\lambda(f)$ of any $f\in \mathcal H_\lambda$ is equal to the range of $f$.
\end{enumerate}
    \label{corSpec1}
\end{corollary}

\begin{proof}
    If $f$ is non-vanishing, then $\delta_x f \neq 0 $ for any $x \in G$, and thus by Theorem~\ref{thmSpec}(i) $f$ does not lie in any maximal ideal of $\mathcal H_\lambda$. As a result, $f$ is invertible. This proves Claim~(i).

    Turning to Claim~(iii), let $f \in \mathcal H_\lambda $ be arbitrary. It is clear that $ \ran f \subseteq \sigma_\lambda(f)$ (since $f-z$ has a zero whenever $z \in \ran f$, and thus cannot have a multiplicative inverse). If $ z \in \sigma_\lambda(f)$ and $ z \notin \ran f$, then $ f - z$ is a nowhere-vanishing non-invertible element of $\mathcal H_\lambda$, which contradicts Claim~(i). Thus, we have $\sigma_\lambda(f) \subseteq \ran f $, and we conclude that  $\sigma_\lambda(f) = \ran f$.

    Finally, to verify Claim~(ii), we recall that every element of a unital Banach-$^*$ algebra with strictly positive spectrum has a square root, which can be chosen to also have strictly positive spectrum; e.g., \cite{Rudin91}*{\S10.30}. Since, by Claim~(iii), $f>0$ has $\sigma_\lambda(f) = \ran f \subset (0,\infty)$, it follows that there exists $g \in \mathcal H_\lambda $ with $\sigma_\lambda(g) \in (0,\infty)$ such that $ f = g^2$. Again by Claim~(iii), $\ran g = \sigma_\lambda(g)$, and thus $g$ is strictly positive.  
\end{proof}

\begin{corollary}
    \label{corSpec2}
    The RKHA $\mathcal H_\lambda$ is (i) semisimple; and (ii) symmetric.
\end{corollary}

\begin{proof}
    Claim~(i) follows from the fact that the Gelfand transform $\Gamma_\lambda$ has trivial kernel (by Theorem~\ref{thmSpec}(ii)). For Claim~(ii) we use Corollary~\ref{corSpec1}(iii) in conjunction with the fact that $f^* f \equiv \bar f f \geq 0 $  to conclude that $\sigma_\lambda(f^* f) = \ran (f^* f) \subseteq [0, \infty ) $.
\end{proof}

Next, we consider the state space, $ \mathcal S(\mathcal H_\lambda)$, of a unital RKHA $\mathcal H_\lambda$, i.e., the set of (automatically continuous) positive functionals $ \varphi: \mathcal H_\lambda \to \mathbb C$, normalized such that $ \varphi(1_G) = 1$. By~\eqref{eqRep}, for a unital RKHA $\mathcal H_\lambda$ with reproducing kernel $k$, each nonzero evaluation functional is a state with operator norm equal to $ \sqrt{k(x,x)} $. It should be noted that because we allow continuity constants $C$ different from 1 in our definition of Banach algebras in~\eqref{eqRKHA}, the elements of $ \mathcal S(\mathcal H_\lambda)$ need not have unit operator norm (which would be the case if $C =1$). 

Suppose now that the evaluation functional $\delta_x$ at every $x\in G$ is nonzero (a condition that holds iff $\lambda > 0$). Then, viewing $\delta_x$ as a Dirac probability measure in $ \mathcal P(G)$ leads to the identity
\begin{equation}
    \label{eqIdDelta}
    \delta_x = \langle R(\delta_x), \cdot \rangle_{\mathcal H_\lambda}, \quad \forall x \in G, 
\end{equation}
where $ R : \mathcal P(G) \to \mathcal H_\lambda$ is the kernel mean embedding of probability measures defined in~\eqref{eqKME}. By continuity of the feature map  $ x \mapsto F(x) \equiv k(x,\cdot)$ as a map from $G$ into $\mathcal H_\lambda$, \eqref{eqIdDelta} extends to a map $ P : \mathcal P(G) \to \mathcal S(\mathcal H_\lambda)$ such that
\begin{equation}
    \label{eqPR}
    ( P \nu ) f = \int_G \delta_x f  \, d\nu(x) =  \langle R( \nu ), f \rangle_{\mathcal H_\lambda}.
\end{equation}

Similarly, to each Dirac probability measure $\delta_x \in \mathcal P(G)$, we can assign a state $ \rho_x \in \mathcal S(B(\mathcal H_\lambda))$ of the $C^*$-algebra $B(\mathcal H_\lambda)$ given by
\begin{equation}
    \label{eqRhoX}
    \rho_x = \tr( \Pi_x \cdot ),
\end{equation}
where $\Pi_x : \mathcal H_\lambda \to \mathcal H_\lambda$ is the rank-1 projection operator 
\begin{displaymath}
    \Pi_x f = \frac{\langle k(x,\cdot), f\rangle_{\mathcal H_\lambda}k(x,\cdot)}{k(x,x)} =  \frac{f(x) k(x,\cdot)}{k(x,x)}.
\end{displaymath}
The assignment $ \delta_x \mapsto \rho_x$ in~\eqref{eqRhoX} extends to a map $ Q : \mathcal P(G) \to \mathcal S(B(\mathcal H_\lambda))$ with
\begin{displaymath}
    (Q \nu) A = \int_G \rho_x A  \, d\nu(x).
\end{displaymath}

Intuitively, we can think of states of the non-abelian $C^*$-algebra $B(\mathcal H_\lambda)$ in the range of $Q$ as ``classical'' states induced by Borel probability measures on $G$ (which are states of the abelian Banach $^*$-algebra $\mathcal H_\lambda$). Letting $ \pi : \mathcal H_\lambda \to B(\mathcal H_\lambda)  $ denote the regular representation of $ \mathcal H_\lambda$ with $ \pi(f) g = fg $, the following proposition justifies the interpretation of states in $ \ran Q $ as classical states, in the sense of acting consistently on regular representatives of $\mathcal H_\lambda$ with expectation operators. 

\begin{proposition}
    \label{propStates}With notation as above and under the assumptions of Theorem~\ref{thmSpec}, the following hold.
    \begin{enumerate}[(i)]
        \item The maps $P$ and $Q$ are injective and weak-$^*$ continuous. 
        \item For every $ \nu \in \mathcal P(G)$ and $ f \in \mathcal H_\lambda$ the compatibility relations
            \begin{displaymath}
                \mathbb E_\nu f = P(\nu)( f ) = Q(\nu)( \pi( f ) ) 
            \end{displaymath}
            hold. In particular, we have $Q(\nu)(\pi(f^*)) = \overline{Q(\nu)( \pi (f))}$, even though $ \pi $ need not be a $^*$-homomorphism. 
    \end{enumerate}
\end{proposition}

\begin{proof} 
    Since $G$ is compact and $\mathcal H_\lambda$ is a dense subspace of $C(G)$, the reproducing kernel of $\mathcal H_\lambda$ is $C$-universal and thus characteristic. It follows that $ R $ is injective, and therefore so is $P $ since $R(\nu)$ is the Riesz representative of $P(\nu)$ according to~\eqref{eqPR}. 
    
    For every $ \nu \in \mathcal P(G)$ and $ f \in \mathcal H$ we get
    \begin{align*}
        Q(\nu)(\pi(f)) &= \int_G \tr(\Pi_x \pi(f)) \, d\nu(x) \\
        &= \int_G \sum_{\gamma \in \hat G} \frac{f(x) \langle \psi_\gamma, k(x,\cdot)\rangle_{\mathcal H} \langle k(x,\cdot), \psi_\gamma \rangle_{\mathcal H}}{k(x,x)} \, d\nu(x) \\
        &=  \int_G f(x) \frac{ \sum_{\gamma \in \hat G} \overline{\psi_\gamma(x)} \psi_\gamma(x) }{ k(x,x) } \, d\nu(x) \\
        &= \int_G f(x) \, d\nu(x) = \mathbb E_\nu f = P(\nu)(f), 
    \end{align*}
    where $ \{ \psi_{\gamma} = \sqrt{\lambda(\gamma)} \gamma : \gamma \in \hat G \} $ is the orthonormal basis of $ \mathcal H_\lambda$ from~\eqref{eqn:def:psi}. The above proves the compatibility relations in Claim~(ii), and also implies that $Q $ is injective by the injectivity of $ P $. We also have 
    \begin{displaymath}
        Q(\nu)(\pi(f^*)) = \mathbb E_\nu f^* = \overline{ E_\nu f} = \overline{Q(\nu)(\pi(f))}, 
    \end{displaymath} 
    verifying the $^*$-compatibility relation in Claim~(ii). 
    
    Since $\mathcal H_\lambda$ is a subspace of $C(G)$, the weak-$^*$ continuity of $P$ follows directly from the fact that $P_\nu f = \mathbb E_\nu f$. Similarly, to deduce weak-$^*$ continuity of $Q$, note that for any $A \in B(\mathcal H_\lambda)$, we have  
    \begin{displaymath}
        (Q\nu) A = \mathbb E_\nu f_A, 
    \end{displaymath}
    where the function $f_A : x \mapsto \rho_x A$ is continuous. 
\end{proof}

The map $G \ni x \mapsto \Pi_x$ can be interpreted in an RKHS context as an \emph{operator-valued feature map}. This feature map along with the corresponding embedding $Q$ of probability measures on $G$ generalize the standard RKHS feature maps and kernel mean embeddings of probability measures to the operator-valued setting of $B(\mathcal H_\lambda)$. In other work \cite{GiannakisEtAl22}, we have found these constructions to be useful in the context of quantum computation.

\subsection{\label{secThmSpec}Proof of Theorem~\ref{thmSpec}}

We begin with the following observation about maximal ideals of unital RKHAs. 

\begin{lemma}
    With the assumptions of Theorem~\ref{thmSpec}, every maximal ideal $I$ in $\mathcal H_\lambda$ is orthogonal to the unit $1_G$.  
    \label{lemIdeal}
\end{lemma}

\begin{proof}
    Let $ 1_G = u + v $ with $ u \in I $ and $ v \in I^\perp $. Since $I$ is a proper, closed subspace of $ \mathcal H_\lambda$, the unit $1_G$ does not lie in $I$, and $ v $ is nonzero, i.e., $ 0 < \lVert v \rVert_{\mathcal H_\lambda} \leq 1$. We claim that, in fact, $ \lVert v \rVert_{\mathcal H_\lambda} = 1$. 
    
    To verify this, by rescaling the kernel of $\mathcal H_\lambda$, we assume without loss of generality that the multiplicative constant $C$ in~\eqref{eqRKHA} is equal to 1 (see Remark~\ref{remC}). 
    
    Next, following standard techniques for unital Banach algebras, we equip $\mathcal H_\lambda$ with an equivalent norm, $ \vertiii{\cdot} $, induced from the operator norm of $B(\mathcal H_\lambda)$ and the regular representation $ \pi $, viz. 
    \begin{displaymath}
        \vertiii{ f }= \lVert \pi( f) \rVert_{B(\mathcal H_\lambda)}. 
    \end{displaymath}
    This norm is a Banach algebra norm satisfying 
    \begin{displaymath}
        \vertiii{fg} \leq \vertiii{f}  \vertiii{g}, \quad \vertiii{1_G} = 1, \quad  \vertiii{f}  \leq \lVert f \rVert_{\mathcal H_\lambda}.  
    \end{displaymath}
    In particular, $ \vertiii{v}  \leq \lVert v \rVert_{\mathcal H_\lambda} $, so if 
    \begin{displaymath}
        \lVert v \rVert_{\mathcal H_\lambda} < 1 = \vertiii{1_G}, 
    \end{displaymath}
    then $\vertiii v  < 1$ and  $u =  1- v $ is invertible. 
    
    The latter, implies that $ I $ contains an invertible element, contradicting the fact that it is a maximal, and thus proper, ideal in $\mathcal H$. It follows that $ \lVert v \rVert_{\mathcal H_\lambda} = 1$ and $ u = 0 $, proving that $ 1_G$ lies in $ I^\perp$.     
\end{proof}
    
We now continue with the proof of Theorem~\ref{thmSpec}. 

\subsection*{Claim~(i)} For every $f\in \mathcal H_\lambda$, $ x \in G$, and net $(x_i)$ converging to $x$, we have 
    \[  \beta_{\lambda}(x_i)(f) = f(x_i) \to f(x) = \beta_{\lambda}(x)(f), \]
so $ \beta_{\lambda} $ is weak-$^*$ continuous. In addition, $ \beta_{\lambda} $ is injective since $\beta_\lambda(x) = \delta_x = \langle k(x,\cdot ), \cdot \rangle_{\mathcal H_\lambda}$, and $ k(x,\cdot) $ is the image of $x$ under the feature map $ F : G \to \mathcal H_\lambda$ (which is injective since $k$ is characteristic). Therefore, since $G$ and $ \sigma(\mathcal H_\lambda) $ are compact Hausdorff spaces, to show that $ \beta_\lambda $ is a homeomorphism it suffices to show that it is surjective.  

    To prove the latter by contradiction, suppose that there exists $ \psi \in \sigma(\mathcal H_\lambda) \setminus \ran \beta_{\mathcal H_\lambda}$. Then, $ I := \ker \psi$ is a maximal ideal in $ \mathcal H_\lambda$, which is distinct from $ \ker \delta_x$ for all $ x \in G$. We claim that $ I $ is a dense subspace of $C(G) $. 
    
    To verify this claim, observe first that the closure $ \bar I $ of $I$ in $C(G)$ is an ideal. Indeed, if that were not the case, there would exist $ f \in C(G) $ and $ g \in \bar I $ such that $ f g \notin \bar I$. But since $\mathcal H_\lambda$ is a dense subspace of $C(G)$, there exists a sequence $f_n \in \mathcal H_\lambda$ converging to $ f $ in $C(G) $ norm, and similarly there exists $ g_n \in I $ such that $ g_n \to  g$ in $C(G)$. Defining $ h_n = f_n g_n $, it follows that $ h_n $ is a sequence in $I $ (since $f_n \in \mathcal H_\lambda $, $g_n \in I $, and $ I $ is an ideal in $\mathcal H_\lambda$) with a $C(G)$-norm limit $ h \in \bar I$. The latter is equal to $ f g $, contradicting the assertion that $ fg \notin \bar I $. 
    
    Now suppose that $ \bar I $ were contained in a maximal ideal in $C(G)$. Since the maximal ideal space of $ C(G) $ is in bijective correspondence with the spectrum $ \sigma(C(G))$, there would exist an $ x \in G $ such that $ \bar I \subseteq \ker \delta_x$ (with $ \delta_x $ understood here as an evaluation functional on $C(G)$), contradicting the fact that $I$ is distinct from the kernels of all evaluation functionals on $\mathcal H_\lambda$. We therefore conclude that $ \bar I $ is an ideal distinct from any maximal ideal of $C(G) $, so it must be equal to the whole space  $C(G)$. We have thus verified that $ I $ is dense in $C(G)$.   

    Next, since $I $ is a maximal ideal in $\mathcal H_\lambda$, every $ f \in I $ is $\mathcal H_\lambda$-orthogonal to $1_G$ by Lemma~\ref{lemIdeal}. Moreover, since $ c := \lambda(0_{\hat G}) > 0 $, the integral operator $K: L^2(G) \to \mathcal H_\lambda$ associated with $k$ satisfies $ c^{-1} K 1_G = 1_G$, and we get
\begin{displaymath}
    0 = \langle 1_G, f \rangle_{\mathcal H_\lambda} = c \langle 1_G, f \rangle_{\mathcal H_\lambda} = \langle K 1_G, f \rangle_{\mathcal H_\lambda} = \langle 1_G, K^* f \rangle_{L^2(G)}.
\end{displaymath}
As a result,
\begin{displaymath}
    \lVert 1_G - f \rVert^2_{C(G)} \geq \lVert 1_G - K^* f \rVert_{L^2(G)}^2 = 1 + \lVert K^* f \rVert^2_{L^2(G)} \geq 1,
\end{displaymath}
which contradicts the assertion that $I $ is dense in $C(G)$, proving Claim~(i).

\subsection*{Claim~(ii)} The fact that $\Gamma_\lambda$ coincides with the inclusion map $ \iota$ follows directly from the definition of the former and the fact that $\mathcal H_\lambda$ is a subspace of $C(G)$, viz.,
\begin{displaymath}
    ( \Gamma_\lambda f )(\delta_x) = f(x) =  (\iota f )( x ).  
\end{displaymath}

To verify the claim on the operator norm of $ \Gamma_\lambda$, we use the reproducing property of $ \mathcal H_\lambda $ to get
\begin{align*}
    \lvert ( \Gamma_\lambda f )( \delta_x ) \rvert &=  \lvert f(x) \rvert = \lvert \langle k(x,\cdot), f \rangle_{\mathcal H_\lambda} \rVert \\
    & \leq \lVert k(x,\cdot ) \rVert_{\mathcal H_\lambda} \lvert f \rVert_{\mathcal H_\lambda} \\
    &= \sqrt{ k(x,x ) } \lVert f \rVert_{\mathcal H_\lambda} =\sqrt{l(0_{G})} \lVert f \rVert_{\mathcal H_\lambda}.
\end{align*}
Setting $ f $ to the unit vector $ f = k(x,\cdot) / \lVert k(x,\cdot) \rVert_{\mathcal H_\lambda} $ (which is well defined since $ \delta_x $ is nonzero at any $x \in G$) then saturates the inequality, proving the claim and the theorem. \qed

\section{\label{secMarkov}Reproducing kernel Hilbert algebras and Markov semigroups}

In this section, we study 1-parameter families of RKHAs associated with ergodic Markov semigroups on $L^2(G)$. We assume throughout that $G$ is compact and the Haar measure $\mu$ is normalized to a probability measure.    

Deferring additional details on the relevant theory to one of the many references in the literature, e.g., \cite{ReedSimon75}, we recall that a strongly continuous semigroup $ \{ M_\tau \}_{\tau \geq 0 }$ of operators on $L^2(G)$ is a \emph{Markov semigroup} if for every $ \tau \geq 0 $, $M_\tau$ is positivity-preserving (i.e., $ M_\tau f \geq 0 $, $\mu$-a.e., whenever $ f \geq 0 $, $\mu$-a.e.), $ M_\tau 1_G = 1_G$, and $ \int_G M_\tau f \, d\mu = \int_G f \, d\mu $ for all $ f \in L^2(G)$. Moreover, $ \{ M_\tau \}_{\tau \geq  0 }$ is said to be \emph{ergodic} if $M_\tau f = f$ for all $\tau \geq 0$ implies that $f$ is constant $\mu$-a.e. 

With these definitions, consider a family $\{ \lambda_\tau \in L^1(\hat G) \}_{\tau > 0} $ of functions on the dual group satisfying the conditions
\begin{equation}
    \label{eqMarkov}
    \begin{gathered}
        \lambda_\tau(0_{\hat G}) = 1, \quad \forall \tau > 0, \\
        0 < \lambda_\tau( \gamma ) < 1, \quad \lambda_{\tau}(\gamma) \lambda_{\tau'}(\gamma) = \lambda_{\tau+\tau'}(\gamma), \quad \forall \tau,\tau'>0, \quad \forall \gamma \in \hat G \setminus\{ 0_{\hat G} \}, \\
        \lambda_\tau(\gamma) = \lambda_\tau(-\gamma), \quad \forall \tau > 0, \quad \forall \gamma \in \hat G.
    \end{gathered}
\end{equation}
We let $l_\tau = \hat{\mathcal F} \lambda_\tau $ and $k_\tau(x,y) = l_\tau(x-y)$ be the kernel functions defined as in~\eqref{eqK}, and $\mathcal H_\tau \equiv \mathcal H_{\lambda_\tau}$ the corresponding RKHSs. We also let $\mathcal K_\tau : L^2(G) \to L^2(G)$ be the corresponding kernel integral operators on $L^2(G)$ from Lemma~\ref{lem:Mercer}. Note that $l_\tau$ and $k_\tau$ are real since $\lambda_\tau $ is real and symmetric (i.e., $\lambda_\tau^\star = \lambda_\tau$). 

By~\eqref{eqMarkov}, for each $\gamma \in \hat G$, $\tau \mapsto \lambda_\tau(\gamma) $ is a continuous function, increasing monotonically to 1 as $\tau \to 0^+$.   Consequently, the operators $ \mathcal K_\tau$ converge pointwise to the identity on $L^2(G) $, i.e., $ \lim_{\tau \to 0^+} K_\tau f = f $ for all $ f \in L^2(G)$. Moreover, $ \lVert \mathcal K_\tau \rVert= \lambda_\tau(0_{\hat G}) =  1 $, so $ \{ \mathcal K_\tau \}_{\tau > 0 } \cup \mathcal K_0 $ with $\mathcal K_0 := \Id$ is a strongly continuous contraction semigroup, consisting of self-adjoint compact operators. 

By the Hille-Yosida theorem, there exists a positive, self-adjoint operator $ \mathcal D$ such that, for all $ \tau \geq 0 $,  $ \mathcal K_\tau = e^{-\tau \mathcal D}$. This operator is diagonal in the character basis of $L^2(G)$, i.e., $ \mathcal D \gamma = \eta(\gamma) \gamma $, where $ \eta(\gamma) = - \tau^{-1} \log \lambda_\tau$ for any $ \tau > 0 $. In particular, $ \mathcal D $ has a simple eigenvalue $\eta(0_{\hat G}) = 0$ corresponding to the constant eigenfunction $1_G$. It then follows from results on Markov semigroups (e.g., \cite{DellAntonio16}*{Chapter~14, Theorem~2}) that $-\mathcal D$ is the infinitesimal generator of an ergodic Markov semigroup, $ \{ e^{-\tau \mathcal D} \}_{\tau \geq 0 }  $. 

By construction, the operators $ \mathcal K_\tau$ are identical to $e^{-\tau \mathcal D}$, which implies that for $ \tau > 0$, $ k_\tau(x,\cdot)$ is a transition probability density with respect to Haar measure. That is, we have 
\begin{displaymath}
   k_\tau(x,\cdot) \geq 0, \quad \int_G k_\tau( x, \cdot ) \, d \mu = 1, \quad \forall \tau > 0, \quad \forall x \in G. 
\end{displaymath}

The following theorem provides necessary and sufficient conditions for the spaces $\mathcal H_\tau$ to have RKHA structure. 

\begin{theorem}
    \label{thmMarkov}
    Suppose that the functions $\lambda_\tau \in L^1(\hat G)$ satisfy the Markov properties in~\eqref{eqMarkov}. Then, the corresponding RKHSs $\mathcal H_\tau$ are RKHAs iff the $\lambda_\tau$ are subconvolutive for each $\tau > 0$, i.e., 
    \begin{displaymath}
        ( \lambda_\tau * \lambda_\tau )(\gamma ) \leq C_\tau \lambda_\tau(\gamma).
    \end{displaymath}
\end{theorem}

\begin{proof}
    See Section~\ref{secProofMarkov}. \phantom\qedhere
\end{proof}

The subexponential functions from~\eqref{eqSubexp}, 
\begin{displaymath}
    \lambda_\tau(\gamma) = e^{-\tau \lvert \gamma \rvert^p}, 
\end{displaymath}
are a concrete example satisfying the assumptions of Theorem~\ref{thmMarkov} for the $d$-torus, $G = \mathbb T^d$, $\hat G = \mathbb Z^d$. In the case of the circle, $\mathbb T^1$, the Markov generator $ \mathcal D $ is a fractional diffusion operator given by the $p/2$-th power of the Laplacian, $\mathcal D = \Delta^{p/2}$.

\subsection{\label{secProofMarkov}Proof of Theorem~\ref{thmMarkov}}

The ``if'' part of the theorem follows directly from Theorem~\ref{thmMain}. To prove the ``only if'' part, suppose that $\{ \mathcal H_\tau \}_{\tau > 0 }$ is a 1-parameter family of RKHAs associated with the functions $\{ \lambda_\tau \}_{\tau>0}$ satisfying~\eqref{eqMarkov}. We show that the $\lambda_\tau$ are subconvolutive. To that end, letting $\xi_\tau = \lambda_\tau^{1/2} = \lambda_{\tau/2}$, we make use of the following result. 

\begin{lemma}
    \label{lemXi}
    If $\mathcal H_\tau$ is a Banach algebra under pointwise multiplication of functions, then for every $\hat u, \hat v \in L^2(\hat G)$ there exists (a unique) $\hat w \in L^2(\hat G) $ such that
    \begin{displaymath}
        \xi_\tau \hat w = (\xi_\tau \hat u) * (\xi_\tau \hat v).
    \end{displaymath}
\end{lemma}

\begin{proof}
    By Lemma~\ref{lem:U}, there exist (unique) $f, g \in \mathcal H_\tau$ such that $\mathcal F f = \xi_\tau \hat u$ and $\mathcal F g = \xi_\tau \hat v$. Moreover, $fg$ lies in $\mathcal H_\tau$, so again by Lemma~\ref{lem:U} there exists $\hat w \in L^2(\hat G)$ such that $\mathcal F(fg) = \xi_\tau \hat w$. The claim follows from the fact that $\mathcal F(fg) = (\mathcal F f) * (\mathcal F g) $. 
\end{proof}

Since $G$ is compact, for any $\epsilon > 0 $ we have $ \xi_\epsilon \in L^1(\hat G) \subseteq L^2(\hat G) $, so setting $\hat u = \hat v = \xi_\epsilon$ in Lemma~\ref{lemXi}, it follows that there exist $\hat w_\epsilon \in L^2(\hat G) $ such that 
\begin{displaymath}
    \xi_\tau \hat w_{\epsilon} = \xi_{\tau+\epsilon} * \xi_{\tau+\epsilon}. 
\end{displaymath}
In particular, we have 
\begin{displaymath}
    \hat{\mathcal F}(\xi_{\tau+\epsilon} * \xi_{\tau+\epsilon}) = (\hat{\mathcal F} \xi_{\tau+\epsilon})  (\hat{\mathcal F} \xi_{\tau+\epsilon}) = l_{(\tau+\epsilon)/2}^2.  
\end{displaymath}

Note now that for each $\gamma \in \hat G$, $ \epsilon \mapsto (\xi_{\tau+\epsilon} * \xi_{\tau+\epsilon})(\gamma) $ is a continuous function that increases monotonically as $\epsilon \to 0^+$ to $ (\xi_{\tau} * \xi_{\tau})(\gamma) $. As a result, the family $\{ \hat w_\epsilon \}_{\epsilon > 0 }$ is bounded in $L^2(\hat G)$ norm and thus in $L^\infty(\hat G)$ norm (since $ \lVert \cdot \rVert_{L^\infty(\hat G)} \leq \lVert \cdot \rVert_{L^2(\hat G)}$ by compactness of $G$) by $ \lVert \xi_\tau * \xi_\tau \rVert_{L^2(\hat G)}$. 

By the above, for every $\tau > 0 $, there exists a constant $\tilde C_\tau$ such that for every $\epsilon > 0 $ and $\gamma \in \hat G$ we have
\begin{displaymath}
    \hat w_\epsilon( \gamma ) = \frac{(\xi_{\tau+\epsilon} * \xi_{\tau+\epsilon})(\gamma)}{\xi_\tau(\gamma)} \leq \tilde C_\tau.
\end{displaymath}
Taking the limit $\epsilon \to 0^+$, we obtain
\begin{displaymath}
    (\xi_\tau * \xi_\tau)( \gamma ) \leq \tilde C_\tau \xi_\tau(\gamma),
\end{displaymath}
and since $\xi_\tau = \lambda_{\tau/2}$ and $\tau$ was arbitrary, we conclude that
\begin{displaymath}
    (\lambda_\tau * \lambda_\tau)(\gamma) \leq C_\tau \lambda_\tau(\gamma),
\end{displaymath}
where $C_\tau = \tilde C_{2\tau}$. This verifies the subconvolutive property of the $\lambda_\tau$. \qed

\section{Fourier--Wermer algebras} \label{secWienerAlg}

Motivated by applications to high-dimensional function approximation, we end the paper with a discussion on the inclusion relationships between RKHAs and Fourier--Wermer algebras on compact abelian groups associated with subconvolutive weights. 

Given a positive weight $w: \hat G \to \mathbb R_+$, consider the space 
\begin{equation*}
    \mathcal A_w = \left\{ f \in L^1(G): \sum_{\gamma \in \hat G} w(\gamma)\lvert \mathcal F f(\gamma) \rvert < \infty \right\},
\end{equation*}
equipped with the norm $\lVert f \rVert_{\mathcal A_w} = \sum_{\gamma \in \hat G}  w(\gamma) \lvert \mathcal F f(\gamma) \rvert $. As mentioned in Section~\ref{sec:intro}, in the case $ G = \mathbb T^d$, spaces in the class $\mathcal A_w$ are Fourier--Wermer algebras that have been employed in high-dimensional ($d \gg 1$) function approximation methods \cites{KolomoitsevEtAl21,NguyenEtAl22}. In more detail, assuming (as we will henceforth do) that $w$ is bounded away from 0, every space $\mathcal A_w$ embeds continuously into the Wiener algebra $\mathcal A(G)$, i.e., the Banach algebra of functions on $G$ with absolutely convergent Fourier series,     
\begin{equation*}
    \mathcal A(G) = \left\{ f \in L^1(G): \sum_{\gamma \in \mathbb Z^d} \lvert \mathcal F f(\gamma) \rvert < \infty \right \}, \quad \lVert f \rVert_{\mathcal A(G)} := \lVert \mathcal F f \rVert_{L^1(\hat G)}. 
\end{equation*}
In particular, $\mathcal A_w$ may be identified with the image of $L^1_w(\hat G)$ under the inverse Fourier operator, $\mathcal A_w = \hat{\mathcal F} (L^1_w(\hat G))$, and can thus  be understood as a Banach space of continuous functions whose regularity depends on the weight function $w$. Here, and in what follows, we let $L^p_w(\hat G)$, $p \in (1,\infty)$, be the Banach space on $\hat G$ equipped with the norm $\lVert \hat f \rVert_{L^p_w(\hat G)} = ( \sum_{\gamma \in \hat G} ( w(\gamma) \lvert \hat f(\gamma)\rvert)^p )^{1/p} $.  

An important problem in numerical analysis is the approximation of functions $f$ in an input space $X$ such as $\mathcal A_w(\mathbb T^d)$ by elements $f_n$ in subspaces of finite dimension, such that for a given $n$ the residual $f - f_n$ has low norm, uniformly over $X$, with respect to a target space $Y$ into which $X$ is continuously embedded (e.g., $L^\infty(\mathbb T^d)$, $L^2(\mathbb T^d)$, or $\mathcal A(\mathbb T^d)$). Typically, the error of such approximations is measured using $s$-numbers for the embedding $\iota: X \hookrightarrow Y$ \cite{Pietsch87}, which can be thought of as generalizations of the singular values of $\iota$ when $X$ and $Y$ are Hilbert spaces and $\iota$ is compact. In high-dimensional applications, of particular interest is the dependence of the $s$-numbers on $d\gg 1$. Intuitively, one seeks to take advantage of the regularity properties of $X$ to alleviate the ``curse of dimension'' suffered by finite-rank approximation of arbitrary elements of the target space $Y$. 

The recent paper \cite{NguyenEtAl22} has shown that for weights $w = w_{s,r}$ in the class of dominating mixed smoothness \eqref{eqWSR} the optimal approximation error from $X = \mathcal A_w(\mathbb T^d)$ scales as $n^{-s} (\log n)^{s(d-1)}$ when the output space is $Y = \mathcal A(\mathbb T^d)$. However, the question of whether $\mathcal A_w(\mathbb T^d)$ has Banach algebra structure is left open. An affirmative answer to that question would present additional opportunities to build approximation schemes that leverage algebra structure; see, e.g., \cite{FeichtingerEtAl_minimal_2007} for an example in the setting of harmonic Hilbert spaces.   

\subsection{Algebra structure of $\mathcal A_w$}

Recall the subconvolutivity condition~\eqref{eqSubconvW} that implies \cite{Kuznetsova06} that $L^p_w(\hat G)$ is a convolution algebra on the dual group. Requiring that this condition holds for $p=1$, 
\begin{equation}
    \label{eqSubconvW1}
    (w^{-1} * w^{-1})(\gamma) \leq C w^{-1}(\gamma), \quad \forall \gamma,\gamma' \in \hat G,
\end{equation}
and using the inverse Fourier operator to pass to the primal group $G$, we can deduce that $\mathcal A_w = \hat{\mathcal F}(L^1_w(\hat G))$ is a Banach algebra under pointwise multiplication. This algebra is a dense subalgebra of the Wiener algebra $\mathcal A(G)$. Moreover, the spectrum of $L^1_w(\hat G)$ contains a homeomorphic image of $G$ \cite{Kuznetsova06}*{Theorem~4}, which implies that $\sigma(\mathcal A_w)$ has the same property.  

Note that the weights $w_{s,r}$ satisfy~\eqref{eqSubconvW1} for any $s \geq 2$ and $ r > 0$. Indeed, in dimension $d=1$ the weight $ \tilde w_{s,r}(\gamma) := (1 + \lvert\gamma\rvert^r)^{s/r} $, $\gamma \in \mathbb Z$, is subadditive and $ \tilde w_{s,r}^{-1}$ lies in $L^1(\mathbb Z)$, so $w_{s,r}^{-1}$ is subconvolutive (see \cite{Feichtinger79}). Since, in any dimension $d\in \mathbb N$, $w_{s,r}^{-1}$ is built up as the product $w^{-1}_{s,r}(\gamma) = \prod_{i=1}^d \tilde w_{s,r}^{-1}(\gamma_i)$ with $\gamma = (\gamma_1, \ldots, \gamma_d)$, it follows that $w_{s,r}$ satisfies~\eqref{eqSubconvW1}, so $\mathcal A_{w_{r,s}}$ is a Banach algebra on $\mathbb T^d$.    

For completeness, we note that the algebra property of $\mathcal A_w$ also holds for subadditive weights (without requiring that $w^{-1}$ lies in $L^1(\hat G)$). In this case, we can additionally deduce that the spectrum $\sigma(\mathcal A_w)$ is homeomorphic to $G$. 

Recall Theorems~3 and~4 in \cite{Brandenburg1975}, which collectively imply that if $\hat G$ is discrete (which is the case here since $G$ is compact), and $w : \hat G \to \mathbb R$ satisfies $w \geq 1$ and~\eqref{eqSubadd}, then $L^1(\hat G) \cap L^\infty_w(\hat G)$ is a Banach convolution algebra with a homeomorphic spectrum to $G$. It has been pointed out to us by Feichtinger~\cite{Feichtinger22} that this result readily generalizes to $\mathcal L_w^p(\hat G) := L^1(\hat G) \cap L^p_w(\hat G)$ with $p \in [1,\infty]$ and the norm $\lVert f \rVert_{\mathcal L^p_w(\hat G)} := \lVert f \rVert_{L^1(\hat G)} + \lVert f \rVert_{L^p_w(\hat G)}$. Indeed, for a discrete group $\hat G$, the subadditivity property~\eqref{eqSubadd} implies
\begin{equation*}
    w(\gamma) \leq C (w(\gamma') + w(\gamma-\gamma')), \quad \forall \gamma,\gamma' \in \hat G,
\end{equation*}
leading to the pointwise estimate
\begin{align*}
    \lvert w(f*g)\rvert(\gamma) &\leq C\int_{\hat G}(w(\gamma') + w(\gamma-\gamma')) \lvert f(\gamma') g(\gamma-\gamma') \rvert \, d\gamma'\\
                                &\leq C\left((\lvert w f \rvert * \lvert g \rvert)(\gamma) + (\lvert f\rvert * \lvert wg \rvert)\right)(\gamma),
\end{align*}
which holds for every $f,g \in \mathcal L^p_w(\hat G)$ and $\gamma \in \hat G$. From the above, we get
\begin{align}
    \nonumber\lVert f * g \rVert_{L^p_w(\hat G)} &\leq C \left\lVert \lvert wf \rvert * \lvert g \rvert + \lvert f \rvert * \lvert w g \rvert \right\rVert_{L^p(\hat G)} \\
    \nonumber &\leq C \left( \lVert wf \rVert_{L^p(\hat G)} \lVert g \rVert_{L^\infty(\hat G)} + \lVert f \rVert_{L^\infty(\hat G)} \lVert w g \rVert_{L^p(\hat G)} \right)\\ 
    \label{eqLPW} &\leq C \left( \lVert f \rVert_{L^p_w(\hat G)} \lVert g \rVert_{L^1(\hat G)} + \lVert f \rVert_{L^1(\hat G)} \lVert g \rVert_{L^p_w(\hat G)} \right), 
\end{align}
and thus $\lVert f * g \rVert_{L^p_w(\hat G)} \leq 2 C \lVert f \rVert_{\mathcal L^p_w(\hat G)}\lVert g \rVert_{\mathcal L^p_w(\hat G)}$. Moreover, since $\lVert f * g \rVert_{L^1(\hat G)} \leq \lVert f \rVert_{L^1(\hat G)} \lVert g \rVert_{L^1(\hat G)}$, we have,
\begin{equation*}
    \lVert f * g \rVert_{\mathcal L^p_w(\hat G)} \leq (2C + 1) \lVert f \rVert_{\mathcal L^p_w(\hat G)}\lVert g \rVert_{\mathcal L^p_w(\hat G)},
\end{equation*}
and we conclude that $\mathcal L^p_w(\hat G)$ is a Banach convolution algebra (cf.\ \cite{Brandenburg1975}*{Theorem~3}).

Next, it follows from~\eqref{eqLPW} that for any $n\in\mathbb N$
\begin{equation}
    \lVert f^{*2n} \rVert_{L^p_w(\hat G)} \leq 2 C \lVert f^{*n} \rVert_{L^1(\hat G)} \lVert f^{*n} \rVert_{L^p_w(\hat G)} \leq 2 C \lVert f^{*n} \rVert_{L^1(\hat G)} \lVert f^{*n} \rVert_{\mathcal L^p_w(\hat G)}.
\end{equation}
By \cite{Brandenburg1975}*{Theorem~2}, the above is a sufficient condition for $\mathcal L^p_w(\hat G)$ and $L^1(\hat G)$ to have equal spectra, 
\begin{equation}
    \label{eqSpecConv}
    \sigma(\mathcal L^p_w(\hat G)) = \sigma(L^1(\hat G)).
\end{equation}
Therefore, setting $p=1$ in~\eqref{eqSpecConv} (in which case $\mathcal L^1_w(\hat G) = L^1_w(\hat G)$ since $w\geq 1$), and passing to the primal group by Fourier transforms, we obtain
\begin{equation}
    \label{eqWienerSpec}
    \sigma(\mathcal A_w) = \sigma(\hat{\mathcal F}L^1_w(\hat G)) = \sigma(\hat{\mathcal F}L^1(\hat G)) = \sigma(C(G)).
\end{equation}
Thus, we can conclude that the Fourier--Wermer algebras $\mathcal A_w$ associated with subadditive weights have the same spectra as $C(G)$, analogously to the result in Theorem~\ref{thmSpec} for RKHAS $\mathcal H_{\lambda}$ associated with subconvolutive weights. 

\subsection{Embedding relationships with RKHAs} 

We now consider the case where $w$ and $\lambda = w^{-2}$ satisfy~\eqref{eqSubconvW1} and the assumptions of Theorem~\ref{thmMain}
, respectively, so that $\mathcal A_w$ and $\mathcal H_\lambda$ are both Banach algebras. This will be the case, for instance, for $w = w_{s,r}$ with $ s \geq 2$, as well as the weights $w = \lambda_\tau^{-1/2}$ obtained from Markov semigroups as in Section~\ref{secMarkov} (which include the subexponential weights $w = \lambda^{-1/2}$ from~\eqref{eqSubexp} as a special case). Inclusion relationships between $\mathcal A_w$ and $\mathcal H_\lambda$ are particularly interesting when both of these spaces are algebras, as they induce algebra representations via the corresponding multiplication operators.  

First, it is straightforward to deduce that $\mathcal A_w$ embeds continuously into $\mathcal H_\lambda$, and the operator norm of the embedding is equal to 1 \cite{NguyenEtAl22}. Indeed, defining the linear operators $A: \mathcal A_w \to L^1(\hat G)$ and $\tilde B: \mathcal H_\lambda \to L^2(\hat G)$ such that
\begin{gather*}
    A f(\gamma) = w(\gamma) \mathcal F f(\gamma), \quad \tilde B \hat f(x) = \sum_{\gamma \in \hat G} w^{-1}(\gamma) \gamma(x) \hat f(\gamma),
\end{gather*}
it follows by direct calculation that the following diagram commutes,
\begin{equation*}
    \begin{tikzcd} 
        \mathcal A_w \ar[r,"\Id"] \ar[d,"A"] & \mathcal H_\lambda\\
        L^1(\hat G) \ar[r,"\iota"] & L^2(\hat G) \ar[u,"\tilde B"]
    \end{tikzcd},
\end{equation*}
where $\Id$ is the identity map on functions and $\iota$ is the inclusion map. One also verifies that $A$ and $\tilde B$ have unit operator norm (in fact, $\tilde B$ is unitary), and since $\lVert 1_G \rVert_{\mathcal A_w} = \lVert 1_G \rVert_{\mathcal H_\lambda} = 1$, the embedding $ \mathcal A_w \hookrightarrow \mathcal H_\lambda$ has unit operator norm. It then follows that the map $\pi : \mathcal A_w \to B(\mathcal H_\lambda)$, where $\pi f$ is the multiplication operator by $f$, provides a strongly continuous, faithful representation of $\mathcal A_w$ into $B(\mathcal H_\lambda)$.    

Next, turning to embeddings of the opposite direction, in general we cannot expect a continuous embedding of $\mathcal H_\lambda$ into $\mathcal A_w$ since, unless $G$ is finite, $L^1(\hat G)$ is a strict subspace of $L^2(\hat G)$. Nevertheless, under appropriate assumptions, giving up a small amount of regularity $\epsilon$ is sufficient to obtain a continuous embedding of $\mathcal H_{\lambda^{1+\epsilon}}$ into $\mathcal A_w$. Here, we assume that the weight $w$ is chosen such that $\lambda^\epsilon = w^{-2\epsilon}$ satisfies the assumptions of Theorem~\ref{thmMain} for any $\epsilon>0$. This will hold, for instance, for the ``semigroup'' weights $w = \lambda_\tau^{-1/2}$ from Section~\ref{secMarkov} but not the $w_{s,r}$ family~\eqref{eqWSR}. 

For any $\epsilon >0 $ let us define the linear map $D_\epsilon : L^2(\hat G) \to L^1(\hat G)$ such that 
\begin{equation*}
    D_\epsilon \hat f = w^{-\epsilon} \hat f;
\end{equation*}
this map has operator norm $ \lVert D_\epsilon \rVert \leq \lVert w^{-\epsilon} \rVert_{L^1(\hat G)} $. Defining also $\tilde A : L^1(\hat G) \to \mathcal A_w $ and $ B_\epsilon : \mathcal H_{\lambda^{1+\epsilon}} \to L^2(\hat G)$ as 
\begin{gather*}
    \tilde A \hat f(x) = \sum_{\gamma \in \hat G} w^{-1}(\gamma) \gamma(x) \xi(\gamma), \quad B_\epsilon f(\gamma) = w^{1+\epsilon}(\gamma) \mathcal F f(\gamma), 
\end{gather*}
where $\lVert \tilde A \rVert = \lVert B_\epsilon \rVert = 1$ and $B_\epsilon$ is unitary, leads to the following commutative diagram:
\begin{equation*}
    \begin{tikzcd} 
        \mathcal H_{\lambda^{1+\epsilon}} \ar[r,"\Id"] \ar[d,"B_\epsilon"] & \mathcal A_w\\
        L^2(\hat G) \ar[r,"D_\epsilon"] & L^1(\hat G) \ar[u,"\tilde A"]
    \end{tikzcd}. 
\end{equation*}
We thus conclude that the embedding $\mathcal H_{\lambda^{1+\epsilon}} \hookrightarrow \mathcal A_w$ is continuous for any $\epsilon > 0$.

\appendix

\section{Proof of Lemma~\ref{lem:LCA}} \label{sec:proof:LCA}

\subsection*{Claim (i)} To show that $k $ is uniformly continuous, note first that the kernel shape function $ l $ lies in $C_0(G)$, and is thus uniformly continuous. As a result, for every $ \epsilon > 0 $ and $ (x,x') \in G \times G$  there exists a neighborhood $U$ of the identity element of $ G $ such that
\[ \lvert l( x - x' ) - l(z) \rvert < \epsilon, \quad \forall z \in ( x - x' ) +  U. \]
Therefore, defining the open neighborhood $V = \{ (y,y') \in G \times G : y - y' \in U \} $ of the identity of $ G \times G$, we get
\[\lvert l( x - x' ) - l( y - y' ) \rvert = \lvert k(x,x') - k(y,y') \rvert < \epsilon, \quad \forall (y,y') \in (x,x') + V, \]
which proves that $k $ is uniformly continuous.

To show that $\RKHS_\lambda$ is a subspace of $C_0(G)$, note that every $ f \in  \RKHS_\lambda $ is the $\RKHS_\lambda$-norm limit of finite linear combinations of kernel sections of the form $ f_n = \sum_{j=0}^{n-1} c_j k( x_j, \cdot )$, where $ k( x_j, \cdot )$ lies in $C_0(G)$. Moreover, proceeding similarly to \cite{FerreiraMenegatto2013}*{Lemma~2.1}, we have
\[ |f(x)| = \abs{ \langle k(x, \cdot), f \rangle_{\RKHS_\lambda} } \leq \norm{ k(x, \cdot) }_{\RKHS_\lambda} \norm{f}_{\RKHS_\lambda} = \sqrt{k(x,x)} \norm{f}_{\RKHS_\lambda} \leq  \lVert l \rVert_{C_0(G)}^{1/2} \lVert f \rVert_{\RKHS_\lambda}, \]
and thus
\[ \norm{f}_{\sup} \leq \norm{l}_{C_0(G)}^{1/2}  \norm{f}_{\RKHS_\lambda}. \]
The above implies that the Cauchy sequence $f_n \in \RKHS_\lambda \cap C_0(G) $ converging to $ f \in \RKHS_\lambda$ is also Cauchy with respect to $C_0(G)$ norm, so $ f $ lies in $C_0(G)$. This proves Claim~(i).

\subsection*{Claim~(ii)} Since $\lambda \in L^1(\hat G)$ and $l \in L^1(G)$, we have $ l = \hat{\mathcal F} \lambda \in L^1(G) \cap C_0(G)$, which implies that for every $x \in G$, $ k( x, \cdot ) = S^x l$ lies in $L^1(G) \cap C_0(G)$ Therefore, for every $f\in L^\infty(G)$, $K f(x) = \int_G k(x,\cdot) f\, d\mu$ exists for every $x\in G$, and we have
    \begin{displaymath}
        \lvert K f(x) \rvert = \left \lvert \int_G k(x,\cdot) f \, d\mu \right \rvert \leq \lVert S^x l \rVert_{L^1(G)} \lVert f \rVert_{L^\infty(G)} = \lVert l \rVert_{L^1(G)} \lVert f \rVert_{L^\infty(G)}.
    \end{displaymath}
This shows that $K $ is well-defined as a bounded linear map from $L^\infty(G)$ to the space of bounded functions on $ G$, as claimed.

Next, we have $\lVert l \rVert^2_{L^2(G)} \leq \lVert l \rVert_{C_0(G)} \lVert l \rVert_{L^1(G)}$, which implies that $l$ lies in $L^2(G)$. Thus, for every $f \in L^2(G)$ we can express $ K f ( x ) $ as the inner product 
\begin{displaymath}
    K f( x ) = \langle k(x,\cdot), f \rangle_{L^2(G)} = \langle S^x l, f \rangle_{L^2(G)}.
\end{displaymath}
Therefore, for any $x, y \in G$, we obtain 
\begin{align*}
    \lvert K f( x) - K f( y )  \rvert^2 &\leq  \lVert S^x l - S^y l \rVert_{C_0(G)} \lVert S^x l - S^y l \rVert_{L^1(G)} \lVert f \rVert_{L^2(G)}^2 \\
    &\leq 2 \lVert S^x l - S^y l \rVert_{C_0(G)} \lVert l \rVert_{L^1(G)} \lVert f \rVert_{L^2(G)}^2,    
\end{align*}
The uniform continuity of $ K f $ then follows by the strong continuity of $S^x$, using a neighborhood of the identity of $ G \times G $ analogous to $V$ in the proof of Claim~(i). 

Finally, the uniform continuity of $ K f $ for $ f \in L^1(G)$ follows from a similar argument using the bound
\begin{displaymath}
    \lvert K f( x) - K f( y ) \rvert \leq  \lVert k( x, \cdot ) - k(y, \cdot ) \rVert_{C_0(G)} \lVert f \rVert_{L^1(G)}.    
\end{displaymath}
This completes the proof of Claim~(ii).

\subsection*{Claim~(iii)} The claim is a direct consequence of~\eqref{eqn:FourierShift} and  the definition of $ k $ in~\eqref{eqK}, viz.
\begin{align*}
    K \gamma( x ) &= \int_G k( x, y ) \gamma( y ) \, d\mu(y) = \int_G S^x l(-y ) \gamma( y ) \, d\mu(y) = \int_G S^x l( y ) \gamma( -y ) \, d\mu(y) \\
    &= \Fourier(S^x l)(\gamma) = \gamma(x) ( \Fourier l )(\gamma) = \gamma(x) \lambda(\gamma).
\end{align*}

\subsection*{Acknowledgments} This research was supported by NSF grant DMS 1854383, ONR MURI grant N00014-19-1-242, and ONR YIP grant N00014-16-1-2649. The authors would like to thank H.\ G.\ Feichtinger for stimulating discussions that led to improvement of the manuscript from an earlier version. 



\end{document}